\documentclass[12pt,reqno,UTF8,twoside]{amsart}
\usepackage{mathrsfs}
\usepackage{amsfonts,amssymb,amsmath,amsthm}
\usepackage{cite}
\usepackage[colorlinks=true,citecolor=red]{hyperref}
\usepackage{tocvsec2}
\usepackage{titletoc}
\usepackage{geometry}
\usepackage{bm}
\usepackage{indentfirst}
\usepackage{graphicx}
\usepackage{float} 
\usepackage{booktabs}
\usepackage{longtable}
\usepackage{subfigure}
\usepackage{stmaryrd}
\usepackage{enumerate}
\usepackage{tikz}
\usepackage{ulem}
\usepackage{color,soul}
\usepackage{setspace}
\usepackage{stmaryrd}
\usepackage[pagewise]{lineno} 
\usepackage{appendix}
\usepackage{cases}
\usepackage{todonotes}
\usepackage{enumitem}

\numberwithin{equation}{section}
\newtheorem{theorem}{Theorem}[section]
\newtheorem{lemma}[theorem]{Lemma}

\newtheorem{proposition}[theorem]{Proposition}
\numberwithin{equation}{section}
\newtheorem{remark}[theorem]{Remark}

\newcounter{thm}

\newcommand{\T}{\mathbb{T}}
\newcommand{\N}{\mathbb{N}}
\newcommand{\R}{\mathbb{R}}
\newcommand{\Z}{\mathbb{Z}}
\newcommand{\E}{\mathbb{E}}
\newcommand{\Pb}{\mathbb{P}}

\tikzstyle{process} = [rectangle, rounded corners, minimum width=4cm, minimum height=1cm, text centered, draw=black, align=center]
\tikzstyle{point} = [coordinate, on grid]
\tikzstyle{arrow} = [->,>=stealth]
\tikzstyle{dasharrow} = [dashed,->,>=stealth]

\hypersetup{linkcolor=blue}

\makeatletter
\@namedef{subjclassname@2020}{\textup{2020} Mathematics Subject
	Classification}
\makeatother

\allowdisplaybreaks

\begin{document}
	
\title[Benjamin-Ono equations with 2D control input]{
The Benjamin-Ono equation with 2D control input: 
approximate controllability and its application}
	
\author[J.-C. Zhao]{Jia-Cheng Zhao}
\address[Jia-Cheng Zhao]{School of Mathematical Sciences, Shenzhen University, 518061, Shenzhen, P.R.China.}
\email{zjc@szu.edu.cn}
	
	
	
\begin{abstract}
We establish the approximate controllability in $L^2$ for the nonlinear Benjamin-Ono equation on torus via two-dimensional control input. Our proof  is based on adaptations of geometric control approach introduced by Agrachev and Sarychev. As an application of this control result, we 
study long-time dynamics of a randomly forced equation. 
It is proved that the 
trajectories are unbounded in Sobolev norms almost surely, when the random force is nondegenerate and statistically periodic in time.
\end{abstract}

\maketitle

\setcounter{tocdepth}{2}
	
\section{Introduction}\label{Section-1}

In the present paper, we consider a nonlinear Benjamin-Ono (BO) equation on the torus $\T=\R/2\pi\Z$, reading
\begin{equation}\label{Control-Problem-0}
\left\{
\begin{array}{ll}
\partial_tu+\mathcal H\partial^2_{x}u+u\partial_xu=\eta,\\
u(0,x)=u_0(x).
\end{array}
\right.
\end{equation}
Here, the unknown $u=u(t,x)$ is real-valued and $\mathcal H$ denotes the Hilbert transform, defined via
$\widehat{\mathcal H u}(k)=-i{\rm sgn}(k)\widehat u(k)$ for $k\in\Z$.
The function $\eta=\eta(t,x)$ denotes an external force that excites only in two directions, i.e.
\begin{equation}\label{source-term}
\eta(t,x)=\eta_1(t)\sin x+\eta_2(t)\cos x, 
\end{equation}
in which the real functions $\eta_1,\eta_2$ play the role of control or random perturbation. 

Our main result for system  (\ref{Control-Problem-0}) is twofold. Considering (\ref{Control-Problem-0}) as a controlled system, we establish the global approximate controllability in the scale of $L^2$. We  also study long-time dynamics via control property of the equation. More precisely, by taking $\eta_1,\eta_2$ as random nondegenerate  perturbations whose law is time-periodic, we show that for each $s>0$, the $H^s$-norm of trajectory $u(t,\cdot)$ blows up almost surely as time goes to infinity. See Theorems \ref{Main-theorem-1} and \ref{Main-theorem-2} below for exact version of these results.

The BO equation arises as a model describing the propagation of internal waves in stratified fluids of great depth \cite{Benjamin-67,Ono-75}. The theory of well-posedness has been well developed in recent decades; see  \cite{MP-12,GKT-23,Tao-04,KLV-24} for some contributions. The  equation has also been studied in many different contexts. For instance, we refer the reader to \cite{GKT-21,GKT-24,Sun-21} for integrability, \cite{IMOS-21,IT-19,G-22} for long-time dynamics and \cite{LO-05,LR-15,LLR-15} for controllability.

As we focus on the controllability for (\ref{Control-Problem-0}), let us give a brief description of the aforementioned literature along this line. Linares and Ortega \cite{LO-05} gave a first result  illustrating exact controllability in any time for the linearized BO equation on $\T$:
\begin{equation*}
\partial_tu_L+\mathcal H\partial^2_{x}u_L=f
\end{equation*}
with control input $f=f(t,x)$ in the form
\begin{equation}\label{control-form}
f=\mathcal Gh:=g(x)\left(
h(t,x)-\int_0^{2\pi}g(y)h(t,y)dy
\right),
\end{equation}
where $g$ is a smooth nonnegative function. For the full BO equation 
$$
\partial_tu+\mathcal H\partial^2_{x}u+u\partial_xu=f,
$$
the local exact controllability in any time was established by Linares and Rosier \cite{LR-15} in the scale of $H^s\,(1/2<s\leq 2)$. The local result was then extended to $L^2$ by Laurent, Linares and Rosier \cite{LLR-15}, while the authors also obtained the global exact controllability in large time by proving a stabilization result.

The control $f$ given by (\ref{control-form}) is allowed to be localized on $\T$, and may act directly on all Fourier modes of the system. By contrast, much less seems to be known concerning the setting of finite-dimensional control.
Our aim is to establish the global approximate controllability in any time for the BO equation (\ref{Control-Problem-0}), by  control input involving only two Fourier modes.

The main result of this paper is stated as follows. Let us write $L^2_0$ be the space of those functions $f\in L^2(\T)$ of mean zero, i.e. $$[f]:=\frac{1}{2\pi}\int_0^{2\pi}f(x)dx=0.$$ The usual $L^2$-norm is denoted by $\|\cdot\|$.

\begin{theorem}[Global approximate controllability]\label{Main-theorem-1}
Let $T>0$ and $u_0,u_1\in L^2_0$ be arbitrarily given. Then for every $\varepsilon>0$, there exist piecewise constant control laws $\eta_1,\eta_2\colon[0,T]\rightarrow \R$ such that 
$$
\|u(T)-u_1\|<\varepsilon,
$$
where $u\in C([0,T];L^2_0)$ is the solution of \eqref{Control-Problem-0},\eqref{source-term}.
\end{theorem}

\begin{remark}\label{changevariable-1}
We consider here the mean-zero solutions of \eqref{Control-Problem-0},\eqref{source-term} for the sake of simplicity. In fact, one can obtain
the approximate controllability for \eqref{Control-Problem-0} in the ``mass-conserved'' sense, i.e. the property stated in Theorem~{\rm\ref{Main-theorem-1}} remains true whenever $u_0,u_1\in L^2(\T)$ with $[u_0]=[u_1]$. This can be done by making the change of variable $\tilde u(t,x)=u(t,x-t[u_0])$ {\rm(}cf.~{\rm\cite[Remark~1.1]{LLR-15})}.
\end{remark}

The proof of Theorem \ref{Main-theorem-1} follows from adaptations of the idea by Nersesyan \cite{N-21}, which provides a variant of geometric control approach introduced by Agrachev and Sarychev~\cite{AS-06,AS-05}. Roughly speaking, one may first establish the approximate controllability to an affine space $u_0+\mathcal H_0$, where $\mathcal H_0$ is the subspace generated by those Fourier modes in which the control input excites.
By induction the approximate controllability can then be ``propagated'' to an increasing sequence of spaces $u_0+\mathcal H_k,k\geq 0$ by means of nonlinearity of the equation. This will be accomplished by an asymptotic property (Lemma \ref{Lemma-asymptotic}) combined with a technique of large controls on small time intervals. Finally, the global approximate controllability follows whenever $\bigcup_{k\geq 0}\mathcal H_k$ is dense in $L^2_0$, which is referred to as ``saturating property'' (Lemma~\ref{Lemma-Saturating}). 

The Agrachev-Sarychev approach is widely used for addressing parabolic equations with additive controls; see, e.g. \cite{S-18,GHM-18,N-21} and references therein. Nevertheless, it is also available for dispersive equations in spite of the absence of parabolic regularization, as the issue of approximate controllability does not involve strong dissipation in high frequency; see, e.g. \cite{S-12,C-23,CXZ-23} for the settings of the Schr\"{o}dinger equation and the KdV equation. 

In addition, recent years have witnessed a considerable interest on controllability problems in the context of bilinear control; see e.g. \cite{P-24,DPU-25,DN-25,BP-25,BCP-25}. It is an interesting problem for further research whether similar results to Theorem \ref{Main-theorem-1} are valid for the periodic BO equation with bilinear control.

As an application of control result, we study long-time behavior of the randomly forced BO equation. Assume that $\eta_l,l=1,2$ are random noises that are statistically $T$-periodic, i.e.
\begin{equation*}
\eta^\omega_l(t)=
\eta_{l,k}^\omega(t-kT),\quad t\in [kT,(k+1)T),\,k\in\N,
\end{equation*}
where $\{\eta_{l,k};l=1,2,k\in\N\}$ is a sequence of i.i.d.~$L^2(0,T)$-valued random variables. Each $\eta_{l,k}$ is of the form
\begin{equation*}
\eta_{l,k}^\omega(t)=\sum_{j=1}^\infty b_j\xi_{l,k,j}^\omega e_j(t),
\end{equation*}
where $b_j$ are nonzero real numbers with $\sum_{j=1}^\infty b_j^2<\infty$, $e_j$ constitutes an orthonormal basis of $L^2(0,T)$, and $\xi_{l,k,j}^\omega$ are independent real-valued random variables having a continuous density $\rho_j$ such that
\begin{equation*}
\int_{-\infty}^\infty x^2\rho_j(x)dx= 1,\quad \rho_j(x)>0,\forall x\in \R.
\end{equation*}

\begin{theorem}\label{Main-theorem-2}
Let $s\in(0,1]$ be arbitrarily given. With the above setting, the trajectory of \eqref{Control-Problem-0},\eqref{source-term} for $u_0\in H^s(\T)$ is unbounded in $H^s(\T)$ almost surely.
\end{theorem}

Our interest on construction of unbounded solutions via random perturbations is inspired by the work of Duca and Nersesyan \cite{DN-25}, where the authors obtained unbounded solutions for the nonlinear Schr\"{o}dinger equation with random potential. Roughly speaking, the control result (Theorem \ref{Main-theorem-1}) means the existence of control input driving the solution in time $T$ to a state with arbitrarily large norm. This, together with continuity of solution map and the nondegeneracy of the random noise, implies 
$$
\inf_{\|u_0\|_{_{H^s}}\leq M}\Pb\{\|u(T)\|_{_{H^s}}>M\}>0
$$
for given $M>0$. Finally, invoking the Markov property and the the Borel-Cantelli lemma, we arrive at 
$$
\Pb\{\|u(kT)\|_{_{H^s}}>M\text{ for some }k\in\N\}=1.
$$
The desired result then follows.

In the unperturbed case (i.e. $\eta\equiv 0$), the authors in \cite{GKT-23,GKT-21} provided deep insights in the long-time dynamics of the periodic BO equation, by means of the Birkhoff coordinates. Their results show that the solutions are almost periodic and the set of orbit is relatively compact.

\medskip

This paper is organized as follows. Section \ref{Section-2} includes the results on global well-posedness for (\ref{Control-Problem-0}) and its generalization. Especially, the continuity of solution map is fundamental to our proof of the main results. In Section \ref{Section-3}, we shall establish the result of global approximate controllability (Theorem \ref{Main-theorem-1}) for system (\ref{Control-Problem-0}),(\ref{source-term}). Finally, we investigate the random dynamics stated in Theorem \ref{Main-theorem-2} in Section \ref{Section-4}, as an application of the control result.

\medskip

\noindent {\bf Notation.} In the present paper we use the following notation.

\begin{itemize}
\item Define $\widehat{f}(k)$ to be the Fourier transform on $\Z$ for a $2\pi$-periodic function, so that 
$f(x)=\sum_{k}\widehat{f}(k)e^{ikx}$.
We also write 
\begin{equation*}
P_0f=[f]=\widehat{f}(0),\quad P_{\pm}f=\sum_{k= 1}^\infty\widehat{f}(\pm k)e^{\pm ikx}.
\end{equation*}
Similarly, $P_{N}$ and $P_{\geq N}$ stand for the projections on the Fourier modes of $|k|< N$ and $k\geq N$, respectively.
In addition, $\widehat{\partial_x^{-1}f}(k)=\frac{1}{ik}\widehat{f}(k)$ for $f\in L^2_0(\T)$ and $k\in\Z\setminus\{0\}$.

\item For $s\geq 0$ and $c\in\R$, we introduce the affine space $$H^s_c=H^s(\T)\cap \{[f]=c\}.$$
The usual $H^s$-norm is denoted by $\|\cdot\|_s=\|J^s\cdot\|$ with  $J^s=(1-\partial_x^2)^s$. 

\item 
We write $L^p_TB_x=L^p(0,T;B)$ for any space $B$ of functions in $x$.

\item Let $S(t),t\in\R$ be the group of bounded linear operators generated by $-\mathcal H\partial^2_{x}$ on $H^s_0$.

\item 
For $s,b\in\R$, the Bourgain spaces $X^{s,b},Z^{s,b}$ (under the BO setting) is defined as
$$
\|f\|_{_{X^{s,b}}}^2=\|\langle\tau+k|k|\rangle^{b}\langle k\rangle^{s}\widehat f\|^2_{_{L^2_{\tau,k}}},\quad
\|f\|_{_{Z^{s,b}}}^2=\|\langle\tau+k|k|\rangle^{b}\langle k\rangle^{s}\widehat f\|_{_{L^2_k L^1_\tau}}^2.
$$
\item The space $Y^s$, defined via
$$
\|f\|_{_{Y^s}}=\|f\|_{_{X^{s,1/2}}}+\|f\|_{_{ Z^{s,0}}},
$$
is contained in the space $C(\R;H^s)$ of continuous functions. 

\item
For $T>0$, the restricted spaces on $[0,T]$ is given by
$$
\|f\|_{_{B_T}}=\inf\{\|\tilde f\|_{_{B}};\tilde f|_{_{[0,T]\times\T}}=f \},
$$
where $B=X^{s,b}, Z^{s,b}$ or $Y^s$. Let us set
$$
X^s_T=C([0,T];H^s_0)\cap L^4_TW^{s,4}_x\cap X^{s-1,1}_T.
$$

\item For two curves $f\colon[0,T_1]\rightarrow L^2_0$ and $g\colon[0,T_2]\rightarrow L^2_0$, the notation $f*g$ stands for their concatenation:
$$
(f*g)(t)=\begin{cases}
f(t),\quad t\in [0,T_1],\\
g(t-T_1),\quad t\in(T_1,T_1+T_2].
\end{cases}
$$

\item The letter $C$ denotes a positive generic constant that does not depend on the data of solutions, and $C(b_1,b_2,\cdots,b_k)$ is increasing  in each parameter $b_i$. 

\item For $a\in \R$ we denote by $a^{\pm}$ the constant of the form $a\pm \varepsilon$ with some $0<\varepsilon<<1$.

\end{itemize}

\section{Cauchy problem of the BO equation}\label{Section-2}

This section includes some basic facts on global well-posedness of (\ref{Control-Problem-0}) and its generalization. These results may be well accepted and are reproduced here for easy reference.

We in the sequel consider the following equation
\begin{equation}\label{generalized-problem}
\left\{
\begin{array}{ll}
\partial_tu+\mathcal H\partial^2_{x}(u+\zeta)+(u+\zeta)\partial_x(u+\zeta)=g,\\
u(0,x)=u_0(x),
\end{array}
\right.
\end{equation}
where $\zeta\in C^\infty(\T)\cap L^2_0$ and $g\in L^2(0,T;H^s_0)$. Evidently,  problem (\ref{generalized-problem}) coincides with (\ref{Control-Problem-0}) when $\zeta\equiv 0$.

While the nonlinearity of the BO equation prohibits a direct application of fixed-point arguments, Tao's gauge transform \cite{Tao-04} is widely used for improving the nonlinearity. 
More precisely, we introduce $F=\partial_x^{-1}(u+\zeta)$ and $G=\partial_x^{-1}g$. 
The equation for $F$ reads
\begin{equation*}
\partial_t F+\mathcal H\partial^2_{x}F+\tfrac{1}{2}(\partial_xF)^2=\tfrac{1}{2}P_0(\partial_xF)^2+G.
\end{equation*}
After elementary calculations, it follows that $W:=P_+e^{\frac{i}{2}F}$ satisfies 
\begin{equation*}
\partial_tW-i\partial^2_{x}W=\tfrac{i}{2}P_+[e^{\frac{i}{2}F}(\tfrac{1}{2}P_0(\partial_xF)^2+G-2iP_-\partial_x(u+\zeta))].
\end{equation*}
The equation for $w:=\partial_xW$ (that is referred as to the gauge transform) is in the form
\begin{equation}\label{Problem-Gauge}
\begin{aligned}
&\partial_tw-i\partial^2_{x}w-\tfrac{i}{4}P_0((u+\zeta)^2)w\\
&=
-\tfrac{1}{4}P_+((u+\zeta)Ge^{\frac{i}{2}F})+\tfrac{i}{2}P_+(ge^{\frac{i}{2}F}) +\partial_xP_+(W\cdot P_-\partial_x(u+\zeta)).
\end{aligned}
\end{equation}
The first two terms in the RHS of (\ref{Problem-Gauge}) behaves well, while the last could be addressed by a bilinear estimate in Bourgain spaces. As a consequence, we have the following result.

\begin{proposition}\label{Prop-GWP}
Let $T>0,s\in[0,1]$ and $\zeta\in C^\infty(\T)\cap L^2_0$. Then for every $u_0\in H^s_0$ and $g\in L^2(0,T;H^s_0)$, problem \eqref{Control-Problem-0} admits a  solution $u$ that is unique in the class
$$
u\in X^s_T\quad\text{with }w=\partial_xP_+e^{\frac{i}{2}F}\in Y^s_T.
$$
Moreover, the solution map $$H^s_0\times L^2(0,T;H^s_0)\ni (u_0,g)\mapsto u\in C([0,T];H^s_0)$$ is Lipschitz continuous on bounded sets. 
\end{proposition}

\begin{remark}\label{changevariable-2}
The continuity of solution map could be extended to the mass-conserved setting. More precisely, the solution map $(u_0,g)\mapsto u$ is locally Lipschitz from $H^s_c\times L^2(0,T;H^s_0)$ to $C([0,T];H^s_c)$. This can be done by also invoking the change of variable $\tilde u(t,x)=u(t,x-t[u_0])$ {\rm(}as in Remark~{\rm\ref{changevariable-1})}
\end{remark}

\begin{proof}[\bf Proof of Proposition \ref{Prop-GWP}]
We begin with a simple reduction. By using the change of variable $\tilde u=u+\zeta$, one can take $\zeta\equiv 0$ without loss of generality. Moreover, the issue for $s=0$ has been solved by \cite[Theorem 1.3]{LLR-15}. We proceed to consider the case of $0<s\leq 1/4$.

The following a priori estimates for system (\ref{generalized-problem}),(\ref{Problem-Gauge}) in the space $X^s_T\times Y^s_T$ are fundamental to the proof; see Steps 1-3 below. 
By the well-posedness result in $L^2$, one gets that problem \eqref{Control-Problem-0} admits a solution $u\in X_T^0$.

{\it Step 1: $Y^s$-estimate of $w$}. Let $0<\delta<T$.
Making use of the inequality $$
\begin{aligned}
&\|e^{\frac{i}{2}F}-1\|_{1}\leq C\|u\|,\\
&\|f\|_{_{X_\delta^{s,-1/2}}}+\|f\|_{_{ Z^{s,-1}_\delta}}\leq C \delta^{1/2-}\|f\|_{_{L^2_\delta H^s_x}},\\
&\|f_1f_2\|_{s}\leq C\|f_1\|_{1}\|f_2\|_{s},
\end{aligned}
$$
we first obtain that 
\begin{equation}\label{estimate-1}
\begin{aligned}
&\|P_+(uGe^{\frac{i}{2}F})\|_{_{X_\delta^{s,-1/2}}}+\|P_+(uGe^{\frac{i}{2}F})\|_{_{ Z^{s,-1}_\delta}}\\
&\leq C \delta^{1/2-}\|u\|_{_{L^\infty_\delta H^s_x}}\left[1+\|u\|_{_{L^\infty_\delta L^2_x}}\right]\left[\|u\|_{_{L^\infty_\delta L^2_x}}+\|g\|_{_{L^2_\delta H^s_x}}\right].
\end{aligned}
\end{equation}
Similarly, one can deduce that
\begin{equation}\label{estimate-2}
\|ge^{-\frac{i}{2}F}\|_{_{X_\delta^{s,-1/2}}}+\|ge^{-\frac{i}{2}F}\|_{_{ Z^{s,-1}_\delta}}\leq C \delta^{1/2-}\left[1+\|u\|_{_{L^\infty_\delta L^2_x}}\right]
\|g\|_{_{L_\delta^2H^s_x}}.
\end{equation}
For the last term in the RHS of (\ref{Problem-Gauge}), we invoke a bilinear estimate established in \cite[Proposition 3.5]{MP-12} and infer that
\begin{equation}\label{estimate-3}
\|\partial_xP_+(W\cdot P_-\partial_xu)\|_{_{X_\delta^{s,-1/2}}}+\|\partial_xP_+(W\cdot P_-\partial_xu)\|_{_{ Z^{s,-1}_\delta}}\leq C \delta^{1/8-}\|w\|_{_{X_\delta^{s,1/2}}}\|u\|_{_{X^0_\delta}}.
\end{equation}
We substitute (\ref{estimate-1})-(\ref{estimate-3}) and 
$$
\left\|\int_0^t e^{i\partial_{x}^2(t-\tau)}f(\tau)d\tau\right\|_{_{Y^s_\delta}}\leq C\left[\|f\|_{_{X_\delta^{s,-1/2}}}+\|f\|_{_{Z^{s,-1}_\delta}}
\right]\quad \text{with }P_+f=f
$$ 
into the Duhamel formula for (\ref{Problem-Gauge}). It then follows that
\begin{equation}\label{priori-1}
\begin{aligned}
\|w\|_{_{Y^s_\delta}}&\leq C(\|u\|_{_{X^0_T}})\left[
\|w_0\|_{_{s}}+\|g\|_{_{L^2_TH^s_x}}
\right]\\
&\quad+C(\|u\|_{_{X^0_T}},\|g\|_{_{L^2_TH^s_x}})\delta^{1/8-}\left[\|u\|_{_{L^\infty_\delta H^s_x}}+\|w\|_{_{X_\delta^{s,1/2}}}\right].
\end{aligned}
\end{equation}

{\it Step 2: $X^{s-1,1}$-estimate of $u$.} Let us invoke a recent result on fractional Leibniz rule (see \cite[Proposition 1]{BOZ-25}):
\begin{equation}\label{estimate-Leibnizrule}
\|J^s(f_1f_2)\|_{_{L^r}}\leq C \left(\|J^sf_1\|_{_{L^{p_1}}}\|f_2\|_{_{L^{q_1}}}+\|f_1\|_{_{L^{p_2}}}\|J^sf_2\|_{_{L^{q_2}}}\right)\footnote{The authors in \cite{BOZ-25} also considered the endpoint cases of $L^\infty$ ($p_j=q_j=r=\infty$) and $L^1$ ($p_1=p_2=1,1\leq q_1=q_2\leq \infty$ and $\frac{1}{2}\leq r\leq 1$).},
\end{equation}
where $1<p_j,q_j\leq \infty\,(j=1,2)$ and $\frac{1}{2}< r<\infty$ such that $\frac{1}{r}=\frac{1}{p_j}+\frac{1}{q_j}$.
Define $\eta_T\in C_0^{\infty}(\R)$ with $\eta_T(t)=1$ for $t\in[-T,T]$ and $\eta_T(t)=0$ for $|t|\geq 2T$. To continue, we construct an appropriate extension of $u$. Let $v(t)=S(-t)u(t)$ on $[0,\delta]$ and $\partial_t v(t)=0$ for $t\in[-2T,2T]\setminus[0,\delta]$. Then $\tilde u(t):=\eta_T(t)S(t)v(t)$ satisfies $\tilde u(t)=u(t)$ for $t\in[0,\delta]$ and 
\begin{equation*}
\begin{aligned}
\|\tilde u\|_{_{X^{s-1,1}}}&\leq C\left[
\|\partial_tv\|_{_{L^2_{[-2T,2T]}H^{s-1}_x}}+\|v\|_{_{L^2_{[-2T,2T]}H^{s-1}_x}}
\right]\\
&\leq C\left[
\|\partial_tv\|_{_{L^2_{\delta}H^{s-1}_x}}+\|v\|_{_{L^2_{\delta}H^{s-1}_x}}
\right].
\end{aligned}
\end{equation*}
Moreover, one has 
\begin{equation}\label{estimate-4}
\|\partial_tv\|_{_{L^2_{\delta}H^{s-1}_x}}=\|S(-t)(\mathcal H\partial_x^2u+\partial_tu)\|_{_{L^2_{\delta}H^{s-1}_x}}= \|-\tfrac{1}{2}\partial_x(u^2)+g\|_{_{L^2_{\delta}H^{s-1}_x}}.
\end{equation}
From (\ref{estimate-Leibnizrule}) we derive that
$$
\|\partial_x(u^2)\|_{s-1}=\|u^2\|_{s}=\|J^s(u^2)\|\leq C \|J^su\|_{_{L^{4}}}\|u\|_{_{L^{4}}}.
$$
This indicates that
\begin{equation}\label{estimate-5}
\|\partial_x(u^2)\|_{_{L^2_\delta H^{s-1}_x}}\leq C \|u\|_{_{L^4_\delta W^{s,4}_x}}\|u\|_{_{L^4_\delta L^{4}_x}}.
\end{equation}
On the other hand, the Duhamel formula for (\ref{generalized-problem}) indicates that
\begin{equation}\label{estimate-6}
\|v\|_{_{L^2_\delta H^{s-1}_x}}=\|u\|_{_{L^2_\delta H^{s-1}_x}}\leq C\left[ \|u_0\|_{s}+\|g\|_{_{L^2_\delta H^{s}_x}}+\|\partial_x(u^2)\|_{_{L^2_\delta H^{s-1}_x}}\right].
\end{equation}
Summarizing (\ref{estimate-4})-(\ref{estimate-6}), we conclude that
\begin{equation}\label{priori-2}
\|u\|_{_{X^{s-1,1}_\delta}}\leq \|\tilde u\|_{_{X^{s-1,1}}}\leq 
C\left[ \|u_0\|_{s}+\|g\|_{_{L^2_T H^{s}_x}}\right]+C(\|u\|_{_{X^0_T}})\|u\|_{_{L^4_\delta  W^{s,4}_x}}.
\end{equation}

{\it Step 3: $L^\infty H^s$- and $L^4W^{s,4}$-estimates of u}. In this step, we shall adopt a technique of frequency analysis. 
Let $N\in\N^+$ to be specified later.
We first observe that 
$$
\|P_Nu\|_{_{L^4_\delta W^{s,4}_x}}\leq  \delta^{1/4}\|P_Nu\|_{_{L^\infty_\delta W^{s,4}_x}}\leq C(N)\delta^{1/4}\|P_Nu\|_{_{L^\infty_\delta H^s_x}},
$$
which indicates that it suffices to estimate $\|P_Nu\|_{_{L^\infty_TH^s_x}}$.
To this end,
it is easy to check that $u_N:=P_Nu$ satisfies
\begin{equation*}
\partial_tu_N+\mathcal H\partial_x^2u_N+\tfrac{1}{2}P_N(\partial_x(u^2))=P_Ng.
\end{equation*} 
Using the Duhamel formula, we obtain that
\begin{equation*}
\|u_N\|_{_{L^\infty_\delta H^s_x}}\leq  C\left[\|u_0\|_{s}+\|g\|_{_{L^2_T H^{s}_x}}\right]+C(N)\delta^{1/2}\|u\|_{_{L^4_\delta L^4_x}}^2.
\end{equation*}
As a consequence, we conclude the following estimate for the low frequency:
\begin{equation}\label{estimate-LF}
\begin{aligned}
\|P_Nu\|_{_{L^\infty_\delta H^s_x}}+\|P_Nu\|_{_{L^4_\delta W^{s,4}_x}}\leq C(N)\left[\|u_0\|_{s}+\|g\|_{_{L^2_\delta H^{s}_x}}\right]+C(N)\delta^{1/2}\|u\|_{_{L^4_\delta L^4_x}}^2.
\end{aligned}
\end{equation}

To deal with the high frequency, we decompose  $u=-2iwe^{-\frac{i}{2}F}+e^{-\frac{i}{2}F}P_{\leq 0}(ue^{\frac{i}{2}F})$, which means
\begin{equation*}
\begin{aligned}
P_{\geq N}u&=-2iP_{\geq N}(we^{-\frac{i}{2}F})+P_{\geq N}(P_{\geq N}(e^{-\frac{i}{2}F})\cdot P_{\leq 0}(ue^{\frac{i}{2}F}))\\
&=:A_N+B_N.
\end{aligned}
\end{equation*}
For the first term, it can be derived that 
\begin{equation*}
\|A_N\|_{_{L^\infty_\delta H^s_x}}\leq C \|we^{-\frac{i}{2}F}\|_{_{L^\infty_\delta H^s_x}}\leq C \|w\|_{_{L^\infty_\delta H^s_x}}\left[1+\|u\|_{_{L^\infty_\delta L^2_x}}\right].
\end{equation*}
In addition, since $0<s\leq 1/4$ one has $H^{1-s}(\T)\hookrightarrow L^\infty(\T)$. Thus,
an application of (\ref{estimate-Leibnizrule}) yields that
\begin{equation*}
\begin{aligned}
\|A_N\|_{_{W^{s,4}}}&\leq C\left[ \|J^se^{\frac{i}{2}F}\|_{_{L^\infty}}\|w\|_{_{L^4}}+\|e^{\frac{i}{2}F}\|_{_{L^\infty}}\|J^sw\|_{_{L^4}}\right]\\
&\leq  C\left[1+\|Je^{\frac{i}{2}F}\|\right]\|J^sw\|_{_{L^4}}\leq C \left[1+\|u\|\right]\|J^sw\|_{_{L^4}}.
\end{aligned}
\end{equation*}
This implies that
\begin{equation*}
\|A_N\|_{_{L^4_\delta W^{s,4}_x}}\leq C \left[1+\|u\|_{_{L^\infty_\delta L^2_x}}\right]\|J^sw\|_{_{L^4_\delta L^4_x}}.
\end{equation*}
In conclusion, 
\begin{equation}\label{estimate-AN}
\|A_N\|_{_{L^\infty_\delta H^s_x}}+\|A_N\|_{_{L^4_\delta W^{s,4}_x}}\leq C \left[1+\|u\|_{_{L^\infty_\delta L^2_x}}\right]\|w\|_{_{Y^s_T}},
\end{equation}
where we have tacitly used $X^{0,3/8}_\delta\hookrightarrow L^4_{\delta}L^4_x$ and $Y^s_\delta\hookrightarrow L^\infty_\delta H^s_x$.

For the second term, we observe that
\begin{equation*}
\begin{aligned}
\|B_N\|_{_{L^\infty_\delta H^s_x}}&\leq  \|P_{\geq N}(e^{-\frac{i}{2}F})\cdot P_{\leq 0}(ue^{\frac{i}{2}F})\|_{_{L^\infty_\delta H^s_x}}\\
&\leq C \|P_{\geq N}(e^{-\frac{i}{2}F})\|_{_{L^\infty_\delta H^{1/2+}_x}}\|P_{\leq 0}(ue^{\frac{i}{2}F})\|_{_{L^\infty_\delta H^{s}_x}}\\
&\leq  \frac{C}{N^{1/2-}}\|e^{-\frac{i}{2}F}\|_{_{L^\infty_\delta H^{1}_x}}\|ue^{\frac{i}{2}F}\|_{_{L^\infty_\delta H^{s}_x}}\\
&\leq \frac{C}{N^{1/2-}} \left(1+\|u\|_{_{L^\infty_\delta L^2_x}}\right)^2\|u\|_{_{L^\infty_\delta H^s_x}}.
\end{aligned}
\end{equation*}
It also follows from (\ref{estimate-Leibnizrule}) that
\begin{align*}
\|B_N\|_{_{W^{s,4}}}&\leq C\left[ \|J^sP_{\geq N}(e^{-\frac{i}{2}F})\|_{_{L^\infty}}\|ue^{\frac{i}{2}F}\|_{_{L^4}}+\|P_{\geq N}(e^{-\frac{i}{2}F})\|_{_{L^\infty}}\|J^s(ue^{\frac{i}{2}F})\|_{_{L^4}}\right]\\
&\leq C\left[ \|J^{s+1/2+}P_{\geq N}(e^{-\frac{i}{2}F})\|\|ue^{\frac{i}{2}F}\|_{_{L^4}}+\|J^{1/2+}P_{\geq N}(e^{-\frac{i}{2}F})\|_{_{L^2}}\|J^s(ue^{\frac{i}{2}F})\|_{_{L^4}}\right]\\
&\leq \frac{C}{N^{1/4-}}\left(1+\|u\|\right)\|J^s(ue^{\frac{i}{2}F})\|_{_{L^4}}.
\end{align*}
In particular, 
$$
\|J^s(ue^{-\frac{i}{2}F})\|_{_{L^4}}\leq \left(1+\|u\|_{_{L^2}}\right)\|u\|_{_{W^{s,4}}}.
$$
Accordingly,
\begin{equation*}
\begin{aligned}
\|B_N\|_{_{L^\infty_\delta H^s_x}}+\|B_N\|_{_{L^4_\delta W^{s,4}_x}}\leq \frac{C}{N^{1/4-\varepsilon}}\left(1+\|u\|_{_{L^\infty_\delta L^2_x}}\right)^2\left(\|u\|_{_{L^\infty_\delta H^s_x}}+\|u\|_{_{L^4_\delta W^{s,4}_x}}\right).
\end{aligned}
\end{equation*}
Combined with (\ref{estimate-AN}), this indicates that
\begin{equation}\label{estimate-HF}
\|P_{\geq N}u\|_{_{L^\infty_\delta H^s_x}}+\|P_{\geq N}u\|_{_{L^4_\delta W^{s,4}_x}}\leq  C(\|u\|_{_{X^0_T}})\left[\|w\|_{_{Y^s_\delta}}+\frac{1}{N^{1/4-\varepsilon}}\|u\|_{_{X^s_\delta}}\right].
\end{equation}
Finally, by (\ref{estimate-LF}) and (\ref{estimate-HF}) we conclude that
\begin{equation}\label{priori-3}
\begin{aligned}
\|u\|_{_{L^\infty_\delta H^s_x}}+\|u\|_{_{L^4_\delta W^{s,4}_x}}&\leq C(N)\left[\|u_0\|_{s}+\|g\|_{_{L^2_T H^{s}_x}}\right]
+C(N,\|u\|_{_{X^0_T}})\delta^{1/2}\|u\|_{_{L^4_\delta L^4_x}}\\
&\quad 
+
C(\|u\|_{_{X^0_T}})\left[\|w\|_{_{Y^s_\delta}}+\frac{1}{N^{1/4-\varepsilon}}\|u\|_{_{X^s_\delta}}\right].
\end{aligned}
\end{equation}

{\it Step 4: completing the proof.} Putting (\ref{priori-1}),(\ref{priori-2}) and (\ref{priori-3}) all together, it follows that
\begin{equation*}
\begin{aligned}
\|u\|_{_{X^s_\delta}}+\|w\|_{_{Y^s_\delta}}&\leq
C(N,\|u\|_{_{X^0_T}})\left[
\|u_0\|_{_{s}}+\|g\|_{_{L^2_TH^s_x}}
\right]
+C(N,\|u\|_{_{X^0_T}})\delta^{1/2}\|u\|_{_{L^4_\delta W^{s,4}_x}}\\
& \quad +
C(\|u\|_{_{X^0_T}},\|g\|_{_{L^2_TH^s_x}})\left[\delta^{1/8-}\left(\|u\|_{_{X^s_\delta}}+\|w\|_{_{X_\delta^{s,1/2}}}\right)+\frac{1}{N^{1/4-\varepsilon}}\|u\|_{_{X^s_\delta}}\right].
\end{aligned}
\end{equation*}
As a consequence, one can choose a  sufficiently large $N$ and a small $\delta$, depending only on $\|u\|_{_{X^0_T}}$ and $\|g\|_{_{L^2_TH^s_x}}$, so  that
$$
\|u\|_{_{X^s_\delta}}+\|w\|_{_{Y^s_\delta}}\leq C(N,\|u\|_{_{X^0_T}})\left[
\|u_0\|_{_{s}}+\|g\|_{_{L^2_TH^s_x}}
\right].
$$
One can repeat the above procedure on the time intervals $[k\delta,(k+1)\delta]\,(k\geq 1)$ up to $[0,T]$; notice that the length $\delta$ is independent of $\|u_0\|_{_{s}}$. In conclusion,
$$
\|u\|_{_{X^s_T}}+\|w\|_{_{Y^s_T}}\leq C(\|u_0\|_{_{s}},\|g\|_{_{L^2_TH^s_x}})\left[
\|u_0\|_{_{s}}+\|g\|_{_{L^2_TH^s_x}}
\right],
$$
leading to the global existence and uniqueness of solution in $H^s$. The Lipschitz continuity of solution map follows by adapting trivially (\ref{priori-1}),(\ref{priori-2}) and (\ref{priori-3}) to Lipschitz-type estimates.

For the case of $1/4<s\leq 1$, the arguments of Steps 1 and 2 are also valid. As for Step~3, some modifications are needed in the high-frequency estimates.
Using (\ref{estimate-Leibnizrule}) again yields 
$$
\begin{aligned}
\|J(we^{-\frac{i}{2}F})\|_{_{L^4}}&\leq C\left[ \|Jw\|_{_{L^4}}\|e^{-\frac{i}{2}F}\|_{_{L^\infty}}+\|w\|_{_{L^\infty}}\|Je^{-\frac{i}{2}F}\|_{_{L^4}}\right]\\
& \leq C \left[1+\|u\|_{_{L^4}}\right]\|Jw\|_{_{L^4}}.
\end{aligned}
$$
Interpolating this with the obvious inequality $\|we^{-\frac{i}{2}F}\|_{_{L^4}}\leq C  (1+\|u\|_{_{L^2}})\|w\|_{_{L^4}}$, we conclude that
$
\|we^{-\frac{i}{2}F}\|_{_{W^{s,4}}}\leq C [1+\|u\|_{_{L^4}}]\|J^sw\|_{_{L^4}}.
$
This combined with the embedding $H^{1/4}\hookrightarrow L^4$ implies
$$
\|A_N\|_{_{L^4_\delta W^{s,4}_x}}\leq C \left[1+\|u\|_{_{L^\infty_T H^{1/4}_x}}\right]\|w\|_{_{X^{s,3/8}_\delta}}.
$$
Similar ideas could be applied to deal with $B_N$. The proof is then complete.
\end{proof}

\section{Approximate controllability}\label{Section-3}

In the remainder of the paper, we denote by $\mathcal R(u_0,\zeta,g)$
the solution map of (\ref{generalized-problem}), and $\mathcal R_t(u_0,\zeta,g)$ stands for the restriction at time $t$. For simplicity we write $\mathcal R(u_0,g)=\mathcal R(u_0,0,g)$ (that is associated with the original BO equation).

The following asymptotic property for the generalized equation (\ref{generalized-problem}) is basic to our proof of Theorem~\ref{Main-theorem-1}.

\begin{lemma}\label{Lemma-asymptotic}
Let $u_0\in L^2_0(\T)$ and $\eta,\zeta\in C^\infty(\T)\cap L^2_0$. Then there holds
$$
\lim_{\delta\rightarrow 0^+}\mathcal R_\delta(u_0,\delta^{-1/2}\zeta,\delta^{-1}\eta)=u_0+\eta-\zeta\partial_x\zeta
$$
in the scale of $L^2$.
\end{lemma}

\begin{proof}[\bf Proof]
We begin with noticing that $u:=\mathcal R(u_0,\delta^{-1/2}\zeta,\delta^{-1}\eta)$ satisfies the equation 
\begin{equation}\label{equation-1}
\left\{\begin{array}{ll}
\partial_tu+\mathcal H\partial^2_{x}(u+\delta^{-1/2}\zeta)+(u+\delta^{-1/2}\zeta)\partial_x(u+\delta^{-1/2}\zeta)=\delta^{-1}\eta,\\
u(0)=u_0.
\end{array}
\right.
\end{equation}
Define $v^\delta(t)=u(\delta t)$ for $t\in[0,1]$ and $\delta\in(0,1)$. In view of (\ref{equation-1}), we derive 
the equation for $v^\delta$:
\begin{equation*}
\begin{aligned}
\partial_tv^\delta&=-\delta\left[
\mathcal H\partial^2_{x}(v^\delta+\delta^{-1/2}\zeta)+(v^\delta+\delta^{-1/2}\zeta)\partial_x(v^\delta+\delta^{-1/2}\zeta)-\delta^{-1}\eta
\right]\\
&=-\delta\left[
\mathcal H\partial^2_{x}v^\delta+v^\delta\partial_xv^\delta\right] -\delta^{1/2}\left[\mathcal H\partial^2_{x}\zeta+
\zeta\partial_xv^\delta+v^\delta\partial_x\zeta
\right]+(\eta-\zeta\partial_x\zeta).
\end{aligned}
\end{equation*}
That is, 
$$
\partial_tv^\delta+\delta\mathcal H\partial^2_{x}v^\delta+\delta v^\delta\partial_xv^\delta+\delta^{1/2}\zeta\partial_xv^\delta+\delta^{1/2}v^\delta\partial_x\zeta=-\delta^{1/2}\mathcal H\partial^2_{x}\zeta+(\eta-\zeta\partial_x\zeta).
$$
Taking $\delta=0$ we obtain the ``limit'' equation
$$
\left\{\begin{array}{ll}
\partial_tv^0=\eta-\zeta\partial_x\zeta,\\
v^0(0)=u_0.
\end{array}\right.
$$
In the sequel, we proceed to prove 
the asymptotic property
\begin{equation}\label{limit}
\|v^{\delta}(1)-v^{0}(1)\|\rightarrow 0\quad \text{as }\delta\rightarrow 0^+,
\end{equation}
which combined with the fact $v^0(1)=u_0+\eta-\zeta\partial_x\zeta$ implies the conclusion of this lemma.

For this purpose, we denote 
$$
q_N=\|(I-P_N)u_0\|+\|(I-P_N)(\eta-\zeta\partial_x\zeta)\|
$$
for $N\in\N^+$. Evidently, one has $q_N\rightarrow 0$ as $N\rightarrow\infty$.
Let us consider the difference $w_N:=v^{\delta}-P_Nv^0$, satisfying 
$$
\begin{aligned}
&\partial_tw_N+\delta\mathcal H\partial^2_{x}(w_N+P_Nv^0)+\delta (w_N+P_Nv^0)\partial_x(w_N+P_Nv^0)+\delta^{1/2}\partial_x[\zeta(w_N+P_Nv^0)]\\
&=-\delta^{1/2}\mathcal H\partial^2_{x}\zeta+(I-P_N)(\eta-\zeta\partial_x\zeta).
\end{aligned}
$$
Taking the inner product with $w_N$ in $L^2$, it follows that
$$
\frac{d}{dt}\|w_N\|^2\leq C(C(N)\delta^{1/2}+q_N)(1+\|w_N\|^2),
$$
leading to
\begin{equation*}
\begin{aligned}
\|w_N(1)\|^2&\leq e^{C(C(N)\delta^{1/2}+q_N)}\left[\|w_N(0)\|^2+C(C(N)\delta^{1/2}+q_N)\right]\\
&\leq e^{C(C(N)\delta^{1/2}+q_N)}\left[q_N^2+C(C(N)\delta^{1/2}+q_N)\right],
\end{aligned}
\end{equation*}
by means of the Gronwall inequality.
For every $\varepsilon>0$ one can take $N$ sufficiently large and then $\delta$ sufficiently small, so that
$$
q_N<\frac{\varepsilon}{2},\quad C(C(N)\delta^{1/2}+q_N)<\min\left\{\log 2,\frac{\varepsilon^2}{4}\right\}.
$$
Accordingly, 
\begin{equation}\label{error-1}
\|v^{\delta}(1)-P_Nv^{0}(1)\|^2<\varepsilon^2
\end{equation}
and 
\begin{equation}\label{error-2}
\|(I-P_N)v^{0}(1)\|\leq q_N<\varepsilon.
\end{equation}
In conclusion, the desired property (\ref{limit}) can be derived from (\ref{error-1}) and (\ref{error-2}).
The proof is complete.
\end{proof}

The asymptotic property stated in Lemma \ref{Lemma-asymptotic} means that one could construct a large control in small time interval, 
driving the system along the new directions of the form $$\eta-\zeta\partial_x\zeta.$$ Accordingly, we define an increasing sequence of subspaces $\mathcal H_j\subset L^2_0(\T),j\geq 0$ via 
\begin{equation*}
\begin{aligned}
&\mathcal H_0=\text{span}\{\sin x,\cos x\},\\
&\mathcal H_j=\text{span}\{
\eta-\sum_{l=1}^N\zeta_i\partial_x\zeta_i;\eta,\zeta_i\in \mathcal H_{j-1}
\},\quad j\geq 1.
\end{aligned}
\end{equation*}
Note that $\mathcal H_0$ involves the same Fourier modes of the control input $\eta$ (see (\ref{source-term})). Furthermore, a saturating property, saying that the vector space $L_0^2$ can be ``filled'' with those directions in $\mathcal H_j,j\geq 1$, is necessary for establishing the global approximate controllability in small time.

\begin{lemma}\label{Lemma-Saturating}
The subspace $\mathcal H_0$ is saturating in the sense that $\bigcup_{j=1}^\infty\mathcal H_j$ is dense in $L^2_0$.
\end{lemma}

Let us remark that the directions $\eta-\zeta\partial_x\zeta$ are produced by the nonlinearity $\partial_x(u^2)$ of the BO equation. Thus, the formulation of saturating property in the present setting is the same as that in the KdV case. So we omit the proof of Lemma \ref{Lemma-Saturating}, and the reader is referred to \cite[Proposition 2.4]{C-23} for a detailed and careful proof.

With the help of Lemmas \ref{Lemma-asymptotic} and \ref{Lemma-Saturating}, we are able to prove the 
small-time
approximate controllability.

\begin{lemma}\label{Lemma-smalltime}
Let $u_0,u_1\in L^2_0(\T)$ be arbitrarily given. Then for every $\varepsilon>0$, there exists a time $\delta>0$ and piecewise constant control laws $\eta_1,\eta_2\colon[0,\delta]\rightarrow \R$ such that 
\begin{equation}\label{control1}
\|u(\delta)-u_1\|<\varepsilon,
\end{equation}
where $u\in C([0,T];L^2_0)$ is the solution of \eqref{Control-Problem-0},\eqref{source-term}.
\end{lemma}

\begin{proof}[\bf Proof]
For every $\eta\in \mathcal H_0$ and $\varepsilon>0$, 
one can apply Lemma \ref{Lemma-asymptotic} with $\zeta\equiv 0$. It then follows that there exists a time $\delta\in(0,1)$ such that
$$
\|\mathcal R_\delta(u_0,\delta^{-1}\eta)-(u_0+\eta)\|<\varepsilon.
$$
This implies the approximate controllability to $u_0+\eta$ via the control $\hat\eta:=\delta^{-1}\eta$.

In what follows we argue by induction. Assume that the approximate controllability to $u_0+\mathcal H_j$ via $\mathcal H$-valued controls holds true. The next task is to establish the controllability to $u_0+\mathcal H_{j+1}$. Without loss of generality we consider $\tilde\eta\in\mathcal H_{j+1}$ of the form
$$
\tilde\eta=\eta-\zeta\partial_x\zeta,\quad \eta,\zeta\in \mathcal H_j.
$$
Applying Lemma \ref{Lemma-asymptotic} and using the fact 
$$
\mathcal R_\delta(u_0,\delta^{-1/2}\zeta,\delta^{-1}\eta)=\mathcal R_\delta(u_0+\delta^{-1/2}\zeta,\delta^{-1}\eta)-\delta^{-1/2}\zeta,
$$
it follows that for every $\varepsilon>0$ there exists a time $\delta_1>0$ such that
\begin{equation*}
\|\mathcal R_{\delta_1}(u_0+\delta_1^{-1/2}\zeta,\hat\eta_1)-(\delta_1^{-1/2}\zeta+u_0+\tilde\eta)\|<\frac{\varepsilon}{2}
\end{equation*} 
with $\hat\eta_1=\delta_1^{-1}\eta$.
By Proposition \ref{Prop-GWP} one can take $r>0$ so that
\begin{equation}\label{approximate-1}
\|\mathcal R_{\delta_1}(u_1',\hat\eta_1)-(\delta_1^{-1/2}\zeta+u_0+\tilde\eta)\|<\frac{\varepsilon}{2},
\end{equation}
whenever $\|u_1'-(u_0+\delta_1^{-1/2}\zeta)\|<r$. By induction hypothesis, there exists a time $\delta_2>0$ and a piecewise constant control $\hat \eta_2\colon [0,\delta_2]\rightarrow\mathcal H$ such that
$$
\|\mathcal R_{\delta_2}(u_0,\hat\eta_2)-(u_0+\delta_1^{-1/2}\zeta)\|<r.
$$
Taking $u_1'=\mathcal R_{\delta_2}(u_0,\hat\eta_2)$ in (\ref{approximate-1}) gives rise to 
\begin{equation}\label{approximate-2}
\|\mathcal R_{\delta_1+\delta_2}(u_0,\hat\eta_2*\hat\eta_1)-(\delta_1^{-1/2}\zeta+u_0+\tilde\eta)\|<\frac{\varepsilon}{2}.
\end{equation} 
Again by the induction hypothesis, there exists a time $\delta_3>0$ and a piecewise constant control $\hat\eta_3\colon [0,\delta_3]\rightarrow\mathcal H$ such that
\begin{equation}\label{approximate-3}
\|\mathcal R_{\delta_3}(\mathcal R_{\delta_1+\delta_2}(u_0,\hat\eta_2*\hat\eta_1),\hat\eta_3)-(\mathcal R_{\delta_1+\delta_2}(u_0,\hat\eta_2*\hat\eta_1)-\delta_1^{-1/2}\zeta)\|<\frac{\varepsilon}{2}.
\end{equation}
Note that $$\mathcal R_{\delta_3}(\mathcal R_{\delta_1+\delta_2}(u_0,\hat\eta_2*\hat\eta_1),\hat\eta_3)=\mathcal R_{\delta}(u_0,\eta)$$ with $\delta=\delta_1+\delta_1+\delta_3$ and $\eta=\hat\eta_2*\hat\eta_1*\hat\eta_3$. Combining (\ref{approximate-2}) and (\ref{approximate-3}) we conclude that
\begin{equation*}
\begin{aligned}
\|\mathcal R_{\delta}(u_0,\eta)-(u_0+\tilde\eta)\|&\leq \|\mathcal R_{\delta}(u_0,\eta)-(\mathcal R_{\delta_1+\delta_2}(u_0,\hat\eta_1*\hat\eta_2)-\delta_1^{-1/2}\zeta)\|\\
&\quad\quad+\|\mathcal R_{\delta_1+\delta_2}(u_0,\hat\eta_1*\hat\eta_2)-(\delta_1^{-1/2}\zeta+u_0+\tilde\eta)\|\\
&<\varepsilon.
\end{aligned}
\end{equation*}
This gives rise to the controllability to $u_0+\mathcal H_{j}$ for each $j\geq 0$.

Finally, the conclusion of this lemma follows from the saturating property given in Lemma~\ref{Lemma-Saturating}. 
\end{proof}

The final step is to derive
the approximate controllability in any time $T>0$.

\begin{proof}[{\bf Proof of Theorem \ref{Main-theorem-1}}]
Let $\varepsilon>0$ be arbitrarily given. Thanks to Lemma \ref{Lemma-smalltime} with $u_0=0$, there exists a small time $\delta_1>0$ and a control $\eta_1\colon[0,\delta_1]\rightarrow\mathcal H$ such that
\begin{equation}\label{control2}
\|\mathcal R_{\delta_1}(0,\eta_1)-u_1\|<\frac{\varepsilon}{2}.
\end{equation}
We extend $\eta_1$ on $[0,T]$ be zero. Then, invoking the continuity of solution map (that has been assured by Proposition~\ref{Prop-GWP}), there exists a constant $r>0$ such that
\begin{equation}\label{control3}
\|\mathcal R_t(u_1',\eta')-\mathcal R_t(0,\eta')\|<\frac{\varepsilon}{2},\quad t\in[0,T]
\end{equation}
whenever $\|u_1'\|\leq r$ and $\|\eta\|_{_{L^2_TL^2_x}}\leq 2\|\eta_1\|_{_{L^2_TL^2_x}}$. Again by Lemma \ref{Lemma-smalltime}, one can construct an $\mathcal H$-valued control $\eta_2$ driving system (\ref{Control-Problem-0}) from $u_0$ to the $r$-neighborhood of $0$ in a small time $\delta_2>0$, i.e.
\begin{equation}\label{control4}
\|\mathcal R_{\delta_2}(u_0,\eta_2)\|<r.
\end{equation}

Let $\eta_3(t)\equiv 0$ for $t\in [0,T-\delta_1-\delta_2]$ and take $\eta'=\eta_1*\eta_3$ in (\ref{control3}). It thus follows that
\begin{equation*}
\|\mathcal R_{T-\delta_2}(u_1',\eta_1*\eta_3)-\mathcal R_{T-\delta_2}(0,\eta_1*\eta_3)\|<\frac{\varepsilon}{2}
\end{equation*}
whenever $\|u_1'\|\leq r$. This together with (\ref{control4}) implies that 
$$
\begin{aligned}
&\|\mathcal R_{T}(u_0,\eta)-\mathcal R_{T-\delta_2}(0,\eta_1*\eta_3)\|\\
&=\|\mathcal R_{T-\delta_2}(\mathcal R_{\delta_2}(u_0,\eta_2),\eta_1*\eta_3)-\mathcal R_{T-\delta_2}(0,\eta_1*\eta_3)\|\\
&<\varepsilon
\end{aligned}
$$
with $\eta=\eta_2*\eta_1*\eta_3$. Using (\ref{control2}) and
the fact 
$$
\mathcal R_{T-\delta_2}(0,\eta_1*\eta_3)=\mathcal R_{\delta_1}(\mathcal R_{T-\delta_1-\delta_2}(0,\eta_3),\eta_1)=\mathcal R_{\delta_1}(0,\eta_1),
$$
we arrive at 
$$
\|\mathcal R_{T}(u_0,\eta)-u_1\|<\varepsilon.
$$
The proof is then complete.
\end{proof}

\section{Unbounded trajectory via random perturbation}\label{Section-4}

We continue to use the random setting in Theorem \ref{Main-theorem-2}. Without loss of generality, let us restrict our attention to mean-zero solutions. Otherwise, the analysis below will also work in the context of $H^s_c\,(c\in\R)$ after trivial modifications. This is because the main materials, i.e. Theorem \ref{Main-theorem-1} and Proposition \ref{Prop-GWP}, remain valid for arbitrarily given mass (see Remarks \ref{changevariable-1} and \ref{changevariable-2}).

We introduce a discrete-time Markov process $u_k,k\in\N$ in $H^s_0\,(s\in[0,1])$ given by 
\begin{equation*}
u_{k+1}=\mathcal R_T(u_k,\eta^k),\quad 
u_0\in H^s_0,
\end{equation*}
where 
$$
\eta^k(t,x)=\eta_{1,k}(t)\sin x+\eta_{2,k}(t)\cos x.
$$
The process is well-defined thanks to Proposition \ref{Prop-GWP}, and the Markov property follows from the setting of random noise.

We denote by $\Pb_{u_0}$ the probability measure associated with the trajectory issued from $u_0$, and by $\E_{u_0}$ the corresponding expectation. The ﬁltration generated by the Markov family $(u_k,\Pb_{u_0})$ is written as $\mathcal F_k$. The reader is referred to \cite[Chapter 1.3.1]{KS-12}) for further description of these notions.

With the above preparations, we are in a position to prove the unboundedness of Sobolev norms for (\ref{Control-Problem-0}).

\begin{proof}[\bf Proof of Theorem \ref{Main-theorem-2}]
Given $M>0$ we define a stopping time
$$
\tau_M=\min \{k\geq 0;\|u_k\|_s>M\},
$$
where $\tau_M=\infty$ if the set is empty. By the arbitrariness of $M$,
the conclusion of this theorem could follow from 
\begin{equation}\label{proba-1}
\Pb_{u_0}\{\tau_M<\infty\}=1,\quad u_0\in H^s_0.
\end{equation}
To verify (\ref{proba-1}), the key step is to check
\begin{equation}\label{proba-2}
p:=\sup_{u_0\in H^s_0}\Pb_{u_0}\{\tau_M>1\}<1,
\end{equation}
or equivalently, 
\begin{equation}\label{proba-3}
\inf_{u_0\in H^s_0 }\Pb_{u_0}\{\tau_M\leq 1\}>0.
\end{equation}
If (\ref{proba-2}) is true, one can invoke the Markov property to infer that for $n\geq 2$,
\begin{equation*}
\begin{aligned}
\Pb_{u_0}\{\tau_M>n\}&=\E_{u_0}\Pb_{u_0}\{\tau_M>n|\mathcal F_n\}\\
&=\E_{u_0}(1_{\{\tau_M>n-1\}}\Pb_{u_{n-1}}\{\tau_M>1\})\\
&\leq p\Pb_{u_0}\{\tau_M>n-1\}.
\end{aligned}
\end{equation*}
Here, $1_A$ denotes the characteristic function of a set $A$.
By iteration it follows that 
$$
\Pb_{u_0}\{\tau_M>n\}\leq p^n.
$$
Therefore, an application of the Borel-Cantelli lemma yields (\ref{proba-1}).

Now, it suffices to verify (\ref{proba-3}). We first note that
$$
\Pb_{u_0}\{\tau_M=0\}=1,\quad \|u_0\|_s>M.
$$
So, the issue is reduced to prove 
\begin{equation}\label{proba-4}
\inf_{u_0\in B_s(M) }\Pb_{u_0}\{\tau_M= 1\}>0,
\end{equation}
where $B_s(M)$ stands for the closed ball in $H^s_0$ centered at the origin and with radius $M$.
Let $u_0\in B_s(M)$ be arbitrarily given. It then follows from Theorem \ref{Main-theorem-1} that there exists a piecewise constant control $\hat\eta\colon [0,T]\rightarrow\mathcal H$ such that
\begin{equation*}
\|\mathcal R_T(u_0,\hat\eta)\|>M.
\end{equation*}
Due to the nondegeneracy of random noise and the continuity of $\mathcal R_T$ in $H^s$ (see Proposition~\ref{Prop-GWP}), there exists a small constant $r>0$ such that
\begin{equation*}
\inf_{\|u_0'-u_0\|<r,[u_0']=0}\Pb\{\|\mathcal R_T(u_0',\eta^0)\|>M\}\geq \Pb\{\|\eta^0-\hat \eta\|_{_{L^2_TL^2_x}}< r\}>0.
\end{equation*}
Then, using the compactness of $B_s(M)$ in $L^2_0$, we obtain
$$
\inf_{u_0\in B_s(M)}\Pb\{\|\mathcal R_T(u_0,\eta^0)\|>M\}>0.
$$
This, together with the fact $$\{\|\mathcal R_T(u_0,\eta^0)\|>M\}\subset\{\|\mathcal R_T(u_0,\eta^0)\|_s>M\}=\{\tau_M= 1\},$$ implies (\ref{proba-4}), as desired. 
\end{proof}

\bigskip

\subsection*{Acknowledgements}

The author would like to thank Professor Shengquan Xiang for bringing to my attention the availability of the Agrachev-Sarychev approach to dispersive equations. This work was initiated when the author was visiting Peking University in March of 2026. He thanks the 
institute for hospitality.

The work of the author was supported by the National Natural Science Foundation 
of China grants 12571474.

\bigskip\bigskip\bigskip

\normalem
\bibliographystyle{plain}
\bibliography{References}

\end{document}